\documentclass[11pt,a4paper]{article}
\usepackage{fullpage}
\usepackage[utf8]{inputenc}
\usepackage{amsmath}
\usepackage{amsfonts}
\usepackage{algorithm} 
\usepackage[noend]{algpseudocode}
\usepackage{mathtools}
\usepackage{comment}
\usepackage{bbm}
\usepackage{graphicx}
\usepackage{amsthm}
\usepackage{amssymb}
\usepackage{fancyhdr}
\usepackage{ifpdf}
\usepackage{titling,authblk}

\usepackage{hyperref}
\hypersetup{
    colorlinks=true,       %
    linkcolor=blue,        %
    citecolor=red,         %
    filecolor=magenta,     %
    urlcolor=cyan,         %
    linktocpage=true
}

\usepackage{csquotes}
\usepackage[maxbibnames=9, backend=biber]{biblatex}

\DeclareMathOperator*{\argmax}{arg\,max}
\DeclareMathOperator*{\argmin}{arg\,min}

\let\nyil\to
\renewcommand{\P}{{\cal{} P}}
\newcommand{\Z}{\mathbb{Z}}
\newcommand{\R}{\mathbb{R}}
\newcommand{\egy}{\mathbbm{1}}
\newcommand{\qedwhite}{\hfill \ensuremath{\square}}

\newtheorem{thm}{Theorem}[section]

\newtheorem{lem}[thm]{Lemma}
\newtheorem{lemma}[thm]{Lemma}

\newtheorem{claim}[thm]{Claim}

\newtheorem{defn}[thm]{Definition}

\newtheorem{prob}[thm]{Problem}

\title{Balanced Submodular Flows}
\author{Alpár Jüttner}
\author{Eszter Szabó}
\affil{{\footnotesize Department of Operations Research, E\"otv\"os Lor\'and University, Budapest, Hungary.\\ Email: \texttt{alpar.juttner@ttk.elte.hu, szaboeszti@student.elte.hu}.}}
\date{}

\algrenewcommand\algorithmicloop{\textbf{repeat}}
\begin{document}

\maketitle

\begin{quote}
{\bfseries Abstract:}
This paper examines the Balanced Submodular Flow Problem, that is the
problem of finding a feasible submodular flow minimizing the
difference between the flow values along the edges. A min-max formula is
given to the problem and an algorithm is presented to solve it using
$\mathcal{O}(m^2)$ submodular function minimizations. Then, these
results are extended to the weighted version of the problem. Finally,
the Balanced Integer Submodular Flow Problem is discussed.
\end{quote}

\begin{quote}
\end{quote}
\vspace{5mm}

\section{Introduction}
Balanced optimization problems aim to find a most
equitable distribution of resources. Several problems have been
analysed in the literature such as the Balanced Spanning Tree Problem
studied by Camerini, Maffioli, Martello, and Toth~\cite{CAMERINI}, by
Longshu Wu~\cite{longshu} and by Ficker, Spieksma, and
Woeginger~\cite{RePEc:spr:eurjco:v:6:y:2018:i:3:d:10.1007_s13675-018-0093-y}
or the Balanced Assignment Problem by Martello, Pulleyblank, Toth, and
de Werra~\cite{MARTELLO}. Duin and Volgenant~\cite{DUIN199143}
proposed a generic scheme for minimum deviation problems, that is also
usable for solving certain balanced optimization problems including
both the above ones. Ahuja proposed a parametric simplex method for
the general Balanced Linear Programming Problem~\cite{AHUJA}.

The Balanced Network Flow Problem introduced by
Scutell\`a~\cite{Scutella} aims to find a feasible network flow
minimizing the difference between the maximum and minimum flow values
along the edges, i.e.\ $\max_{e\in A}f(e)-\min_{e\in A} f(e)$.
Scutell\`a presented an algorithm using Newton's approach that
performs $\mathcal{O}(n^2\log^3(n))$ max-flow computations, where $n$
is the number of nodes. The weighted version of the problem has also
been solved by Scutell\`a and Klinz~\cite{bettina99:_balan}.

This paper examines an extension of the problem above to submodular
flows, where the goal is to find a feasible submodular flow minimizing
the difference between flow values along the edges. A strongly
polynomial algorithm is given for solving the Balanced Submodular
Flow Problem with $\mathcal{O}(m^2)$ submodular function
minimizations, where $m$ is the number of edges in the input graph.
Then, Section~\ref{sec:weighted} examines the Weighted Balanced
Submodular Flow problem, that is the problem of finding a feasible
submodular flow minimizing the difference between the maximum and
minimum weighted flow values along the edges.
We show that an optimal solution can be found by solving
$\mathcal{O}(n^4m^6\log^6(m))$ submodular function minimization
problems. In Section~\ref{sec:integral}, the Balanced Integral
Submodular Flow Problem is introduced and an optimal integral solution
is determined using the fractional optimum.

\section{Preliminaries}

\begin{defn}
  For an underlying set $V$, let $\P(V)$ denote the power set of $V$,
  i.\,e.\ the set of all subsets of $V$. The set functions
  $b,p:\P(V)\nyil \R$ are called \emph{submodular} or
  \emph{supermodular} if
\begin{equation}
    b(X)+b(Y)\geq b(X\cup Y)+b(X\cap Y)
\end{equation}
or
\begin{equation}
    p(X)+p(Y)\leq p(X\cup Y)+p(X\cap Y)
\end{equation}
holds for each subsets $X,Y\subseteq V$, respectively. A function is
called \emph{modular} if it is both sub- and supermodular.
\end{defn}

\begin{thm}[Orlin~\cite{submodmin}]\label{thm:submodminalg}
  Assuming that the value of a submodular function $b$ can be computed
  for any $X\subseteq V$ in time $T$, then the value of $\min\{b(X):X\subseteq V\}$
  can be computed in time $\mathcal{O}(n^5 T + n^{6})$.
\end{thm}

Whenever a submodular function minimization is used as a subroutine,
its running time will be denoted by $\Upsilon$ for the sake on
simplicity.

\begin{defn}
  For a directed graph $G=(V,A)$ and a subset of vertices
  $X\subseteq V$, let $\tilde\varrho(X)$ and $\tilde\delta(X)$ denote
  the set of edges entering $X$ and leaving $X$, respectively. For a
  vector $x\in\R^A$, let
  \begin{equation}
    \varrho_x(X):=\sum_{e\in \tilde\varrho(X)} x(e)
    ,\quad
    \delta_x(X):=\sum_{e\in \tilde\delta(X)} x(e)
    \quad\mbox{and}\quad
    \partial_x(X) := \varrho_x(X)-\delta_x(X).
  \end{equation}
  Furthermore, let
  $\varrho(X)$, $\delta(X)$ and $\partial(X)$ denote the number of
  edges entering $X$, leaving $X$, and their difference, respectively.
\end{defn}

It is straightforward to check that $\varrho_x(X)$ and $\delta_x(X)$
are submodular functions for any nonnegative vector $x$. If
$l,u\in \R^A$ and $l\leq u$, then $ \varrho_u(X)-\delta_l(X)$ is
submodular and $ \varrho_l(X)-\delta_u(X)$ is supermodular.

\begin{defn}
  Let us given a directed graph $G=(V,A)$ and a submodular function
  $b:\P(V)\nyil\R$. A vector $x\in\R^A$ is called a \emph{submodular
    flow} if
\begin{equation}
    \varrho_x(X)-\delta_x(X) \leq b(X)
\end{equation}
holds for each $X\subseteq V$.

For vectors $l,u\in\R^A$, a submodular flow $x$ is called
\emph{$(l,u)$-bounded} if $l\leq x\leq u$.
\end{defn}

\begin{thm}\label{thm:submodflow}
  For lower and upper bounds $l,u\in\R^A$, there exists an
  $(l,u)$-bounded submodular flow if and only if $l\leq u$ and
  \begin{equation}
    \varrho_l(X)-\delta_u(X) \leq b(X)
    \label{eqn:submoddual}
  \end{equation}
  holds for each $X\subseteq V$.
\end{thm}

\begin{thm}[Frank~\cite{feasible_submodflow}]\label{thm:submodflowalg}
  Assume that $b:\P(V)\nyil\R$ is submodular and the value of $b(X)$ can be computed
  in time $T$ for any $X\subseteq V$, then an $(l,u)$-bounded
  submodular flow --- or a set $X$ violating~(\ref{eqn:submoddual}) if
  such a flow does not exist --- can be found in time $\mathcal{O}(n^5T)$.
\end{thm}

For estimating the running time of the proposed algorithms, we will
also use the following known results.

\begin{lem}[Goemans~\cite{radzikfrac}]\label{goemans}
  Let $b\in \mathbb{R}^p$ be a real vector and let
  $x_1,x_2,\dots,x_q\in{\{-1,0,+1\}}^p$ be vectors such that
  \[0<bx_{i+1}\leq \frac{1}{2}bx_{i} \quad \forall i \in \{1,2,\dots,q-1\}\]
  Then $q=\mathcal{O}(p\log(p))$
\end{lem}

\begin{thm}[Erdős--Szekeres~\cite{Erdos_Szekeres}]\label{thm:erdos-szekeres}
  Given integer numbers $r, s$ and a sequence of distinct real numbers
  with length at least $(r-1)(s-1)+1$, there exists a
  monotonically increasing subsequence of length $r$ or a
  monotonically decreasing subsequence of length $s$.

  In the case when $r=s$, every sequence of length $n$ contains a
  monotonically increasing or decreasing subsequence of length
  $\sqrt{n}$.
\end{thm}

\begin{thm}[Goemans, Gupta and Jaillet~\cite{compute_value_gen}]\label{thm:Goemans-Gupta-Jaillet}
  Let $b$ be a submodular function on $V$, $|V|=n$ and assume that
  $b(\emptyset)\geq 0$.  Then the value of
  \begin{equation} \label{eq:delta*}
    \delta^*:= \min \left\{\delta \geq 0 :
      \min_{S\subseteq V}\left(b(S)+\delta a(S)\right)\geq 0\right\}
  \end{equation}
  can be computed by performing $\mathcal{O}(n\log^2n)$ submodular
  function minimization over the set $V$, i.\,e.~in time
  $\mathcal{O}(n^2\Upsilon)$.
\end{thm}

A family $R\subseteq \P(V)$ is called a
ring family if it is closed under taking unions and intersections. For
a given family of subsets $X_1, X_2, \dots, X_k$ of $V$, let
$R(X_1, X_2, \dots X_k)$ denote the ring family generated by
$X_1, X_2, \dots, X_k$. 

\begin{lem}[Goemans et al.~\cite{compute_value_gen}]\label{thm:b_decrease_fast}
  Let $b:\P(V) \rightarrow \mathbb{R}$ be a submodular function with
  $b_{\min}=\min_{S\subseteq V}b(S)\leq 0$. Consider a sequence of
  distinct sets $X_1, X_2, \dots, X_q$ such that
  $b(X_1)=b_{\min}, b(X_2)>-2b_{\min}$, and $b(X_i)\geq 4b(X_{i-1})$
  for $3\leq i \leq q$. Then
  \[b(X_q)>\max\{b(X) : X\in R({X_1,\dots , X_{q-1}})\},\]
  therefore
  $X_q\notin R({X_1,\dots , X_{q-1}})$.
\end{lem}

\begin{lem}[Goemans et al.~\cite{compute_value_gen}]\label{thm:notinringfamily}
  Let $b:\P(V) \rightarrow \mathbb{R}$ be a submodular function and
  $\mathcal{T}\subseteq \P(V)$, where $|V|=n$. Assume that
  $b(S)\leq M$ for all $S\in \mathcal{T}$. Let $X \subseteq V$ be a
  subset, such that
  \begin{equation*}
    b(X)>\frac{n^2}{4}(M-b_{\min})
  \end{equation*} 
  Then $X\notin R(\mathcal{T})$.
\end{lem}

\begin{lem}[Goemans et al.~\cite{compute_value_gen}]\label{thm:chainringfamilies}
  Given a sequence of subsets $X_1, X_2 \dots X_k$ of $V$, define
  $L_i = R({X_1,\dots , X_i})$ for $ 1 \leq i \leq k$, where
  $|V|=n$. Assume that $X_i\notin L_{i-1}$ for all $i>1$. Then
    \begin{equation*}
        k\leq \binom{n+1}{2}+1
    \end{equation*}
\end{lem}

\begin{thm}[Goemans et al.~\cite{compute_value_gen}]\label{thm: submodfunc}
  Let $b:\P(V) \rightarrow \mathbb{R}$ be a submodular function with
  $b_{\min}:=\min_{S\subseteq V}b(S)\leq 0$. Consider a sequence of
  distinct sets $T_1, T_2, \dots, T_q$ such that
  $f(T_1)=b_{\min}$, $b(T_2)>-2b_{\min}$ and $b(T_i)\geq 4b(T_{i-1})$ for
  $3\leq i \leq q$. Then $q\leq \binom{n+1}{2}+1$.
\end{thm}

\section{Balanced Submodular Flows}\label{sec:balanced_submod_flow}
\begin{defn}
  The \emph{spread}\footnote{This expression can be seen under another
    name in the literature..} of a vector $x\in\R^A$ is the value
  \[\sigma(x):=\max_{a\in A}x(a) - \min_{a\in A}x(a)\].
\end{defn}

The \emph{Balanced Submodular Flow Problem} is to find a submodular
flow of minimum spread.
\begin{prob}\label{prb:balflow}
  Let us given a directed graph $G=(V,A)$ and a submodular function
  $b:\P(V)\nyil\Z$. Find
  \begin{equation}
    \sigma^*:=\min\left\{\sigma(x): x\in \R^A\quad
      \varrho_x(X)-\delta_x(X) \leq b(X)\quad\forall X\subseteq V\right\}
  \end{equation}
  along with a minimizing vector $x^*$.
\end{prob}

For the sake of simplicity, we assume that there is no isolated vertex
in the graph $G$ and there exists an unbounded submodular flow,
i.e. $b(U)\geq 0$ holds for every subset $U\subseteq V$ with 
$\varrho(U)=0$ and $\delta(U)=0$.

The following section presents a dual characterization of the value of
the minimum spread and then strongly polynomial algorithms to compute
this value will be given. For technical reasons, the cases of Eulerian
and non-Eulerian graphs are treated separately.

\subsection{Eulerian Graphs}

Throughout this section, a graph $G$ is called \emph{Eulerian} if
$\partial({v})=0$ holds for all $v\in V$. Note that $G$ is not
required to be connected.

Clearly, if $G$ is Eulerian and $x\in\R^A$ is a submodular flow, then
$x+\alpha\egy$ is also a submodular flow for any
$\alpha\in\R$. Therefore the Balanced Submodular Flow Problem reduces
to the problem of finding the minimum value $\sigma^*$ for which there
exists a bounded submodular flow $0\leq x^*\leq \sigma^*\egy$.
Applying Theorem~\ref{thm:submodflow}, $\sigma^*$ is the smallest
nonnegative value for which $b(X)+\sigma^*\delta(X)\geq 0$ holds for
all $X\subseteq V$. In other words, we are looking for the root of the
function
\begin{equation}
  \label{eq:nullaf}
  f(\sigma):=\min\{b(X)+\sigma\delta(X) : X\subseteq V\}
\end{equation}
or $f(\sigma)=0$ if this minimum is negative. 
This immediately gives the following dual characterization of the
minimum spread submodular flows.

\begin{thm}\label{thm:eqnulla}
  Assume that $G$ is Eulerian. Then
  \begin{equation}
    \label{eq:eqnulla}
    \sigma^* =  \max\left\{
      \frac{-b(X)}{\delta(X)} :
      X\subseteq V, \delta(X)>0
      \right\}
  \end{equation}
  or $\sigma^*=0$ if the value of this maximum is negative. 
\end{thm}

Therefore, the problem reduces to a fractional optimization problem,
the optimum of which can be calculated using the standard
Discrete Newton Method~\cite{radzikfrac} outlined in
Algorithm~\ref{alg:nulla}. It is straightforward to see that
$\delta(X_i)$ strictly decreases in every iteration and a standard
argument shows that the final set $X_i$ indeed maximizes the value
$\frac{-b(X)}{\delta(X)}$. Thus we get the following

\begin{thm}
  The value $\sigma^*$ of the minimum spread and the corresponding dual $X$ 
  in Theorem~\ref{thm:eqnulla} are determined by Algorithm~\ref{alg:nulla} 
  within $\mathcal{O}(m)$ iterations. \qedwhite{}
\end{thm}

\begin{algorithm}[!ht]
  \begin{algorithmic}[1]
    \State{} Let $\sigma_1 :=0$.
    \State{} $i:=1$
    \Loop{}
    \State{} Let $X_i:=\argmin\{b(X)+\sigma_i\delta(X) : X\subseteq V\}$
    \If {$b(X_i)+\sigma_i\delta(X_i) \geq 0$}
    \State{} RETURN $\sigma_i, X_i$
    \ElsIf{$\delta(X_i) = 0$}
    \State{} RETURN ``INFEASIBLE''
    \Else{} 
    \State{} $\sigma_{i+1}:=\frac{-b(X_i)}{\delta(X_i)}$
    \EndIf{}
    \State{} $i\longleftarrow i+1$
    \EndLoop{}
  \end{algorithmic}
  \caption{Minimum spread calculation in Eulerian graphs}\label{alg:nulla}
\end{algorithm}

\subsection{Non-Eulerian Graphs}

Now, let us consider the Balanced Submodular Flow Problem in
non-Eulerian graphs. Let $G=(V,A)$ be a directed graph assuming that
there exists a node $v\in V$ such that $\partial({v})\neq 0$.  First,
we analyze a simpler algorithm and then improve this
algorithm. Furthermore, a dual characterization of the minimum spread
submodular flows is also given. More precisely, the following theorem
will be proved.

\begin{thm}\label{thm:eqxy}
  Assume that $G$ is not Eulerian. Then
  \begin{equation}
  \label{eq:eqxy}
    \sigma^* = \max_{X,Y\subseteq V} \left\{
      \sigma(X,Y) :
      \partial(X)\geq 0, \partial(Y)< 0, \varrho(X)+\delta(X)>0 \right\},
  \end{equation}
  where
  \begin{equation}
    \label{eq:sigmaxy}
    \sigma(X,Y):= \frac{b(X)\partial(Y)-b(Y)\partial(X)}
      {\delta(Y)\partial(X)-\delta(X)\partial(Y)}
  \end{equation}
  or $\sigma^*=0$ if the value of this maximum is negative. 
\end{thm}

Note that the condition $\varrho(X)+\delta(X)>0$ in~\eqref{eq:eqxy}
is equivalent to $\delta(Y)\partial(X)-\delta(X)\partial(Y)\neq 0$,
assuming that $\partial(X)\geq 0, \partial(Y)< 0$ hold for
$X,Y$. Therefore, the expression in~\eqref{eq:eqxy} is well-defined.

In order to prove Theorem~\ref{thm:eqxy}, we first show that the
expression above constitutes a lower bound of $\sigma^*$ for any pairs
of sets $X$ and $Y$.

\begin{lem}\label{thm:lbxy}
  Let $X,Y\subseteq V$ such that $\partial(X)\geq 0$,
  $\partial(Y) < 0$ and $\varrho(X)+\delta(X)>0$. Then
  \begin{equation}
    \label{eq:lbxy}
    \sigma^*\geq \sigma(X, Y).
  \end{equation}
\end{lem}

\begin{proof}
  By the definition of $x^*$, there exists a real value $\kappa$ such that
  $\kappa\egy\leq x^*\leq (\kappa+\sigma^*)\egy$.  Using
  Theorem~\ref{thm:submodflow} with $l:=\kappa\egy$ and
  $u:=(\kappa+\sigma^*)\egy$ we get

  \begin{equation}
    \kappa\partial(X) - \sigma^*\delta(X) \leq b(X)
  \end{equation}
  and
  \begin{equation}
    \kappa\partial(Y) - \sigma^*\delta(Y) \leq b(Y).
  \end{equation}
  From these, we get
  \begin{equation}
    b(X)\partial(Y) + \sigma^*\delta(X)\partial(Y)
    \leq \kappa\partial(X)\partial(Y)
    \leq b(Y)\partial(X)+\sigma^*\delta(Y)\partial(X),
  \end{equation}
  therefore
  \begin{equation}
    \sigma^*\big(\delta(Y)\partial(X)-\delta(X)\partial(Y)\big)
    \geq b(X)\partial(Y) - b(Y)\partial(X).
  \end{equation}
  Since $\delta(Y)\partial(X)-\delta(X)\partial(Y)>0$, the equation
  above is equivalent to
  \begin{equation}
    \sigma^*
    \geq
    \frac{b(X)\partial(Y) - b(Y)\partial(X)}
    {\delta(Y)\partial(X)-\delta(X)\partial(Y)}
    = \sigma(X, Y).
   \end{equation}
\end{proof}

\begin{proof}[Proof of Theorem~\ref{thm:eqxy}]
 Let
  \begin{equation}
    \label{eq:kappaxy}
    \displaystyle \kappa(X,Y) := \frac{b(X)\delta(Y)-b(Y)\delta(X)}  
    {\delta(Y)\partial(X)-\delta(X)\partial(Y)}
  \end{equation}
  and let $X$ and $Y$ be a pair of sets maximizing $(\ref{eq:eqxy})$
  and among them the one maximizing
  $\kappa^* = \displaystyle \kappa(X,Y)$.  We show that there exists a
  bounded submodular flow such that
  $\kappa^*\egy \leq x \leq (\kappa^*+\sigma^*)\egy$.
  
  Suppose to the contrary, assume that such a flow does not exist. Then,
  Theorem~\ref{thm:submodflow} states that a set
  $C\subseteq V$ with the property
  \begin{equation*}
    \kappa^*\partial(C)-\sigma^*\delta(C)>b(C)
  \end{equation*}
  must exist.  Note that $\varrho(X)+\delta(X)>0$, since otherwise
  $\kappa(X,Y)=0$ and therefore $b(C)<0$ would hold, thus the
  problem would be infeasible.

  If $\partial(C)<0$ and $\partial(X)>0$, we get
  \begin{align*}
    &\big(b(X)\partial(Y)-b(Y)\partial(X)\big)
      \big(\delta(C)\partial(X)-\delta(X)\partial(C)\big) \\
    =\quad &b(X)\partial(Y)\delta(C)\partial(X)-
             b(X)\partial(Y)\delta(X)\partial(C)-\\
    &b(Y)\partial(X)\delta(C)\partial(X)+b(Y)\partial(X)\delta(X)\partial(C)\\
    =\quad & -b(X)\partial(Y)\delta(X)\partial(C)+
             \partial(X)\big(b(X)\partial(Y)\delta(C)-\\
    &b(Y)\delta(C)\partial(X)+b(Y)\delta(X)\partial(C)\big)\\
    =\quad & -b(X)\partial(Y)\delta(X) \partial(C)+
             \partial(X)\big(\delta(C)\big(b(X)\partial(Y)-\\
    &b(Y)\partial(X)\big)+\partial(C)\big(b(Y)\delta(X)\big)\big) \\
    =\quad & \partial(X)\big(\delta(C)\big(b(X)\partial(Y)-
             b(Y)\partial(X)\big)-\partial(C)
             \big(b(X)\delta(Y)-b(Y)\delta(X)\big)\big) \\
    &-b(X)\partial(Y)\delta(X) \partial(C)+
      \partial(X)\partial(C)b(X)\delta(Y)\\
    =\quad & \big(\delta(Y)\partial(X)-\delta(X)\partial(Y)\big)
             \partial(X)\big(\delta(C)\sigma-\partial(C)\kappa\big)-\\
    &b(X)\partial(Y)\delta(X) \partial(C)+
      \partial(X)\partial(C)b(X)\delta(Y)\\
    <\quad & \big(\delta(Y)\partial(X)
             -\delta(X)\partial(Y)\big)\partial(X)(-b(C))\\
    & -b(X)\partial(Y)\delta(X) \partial(C)
      +\partial(X)\partial(C)b(X)\delta(Y)\\
    =\quad & \big(\delta(Y)\partial(X)-\delta(X)\partial(Y)\big)
             \big(b(X)\partial(C)-\partial(X)b(C)\big),
  \end{align*}
  therefore
  \begin{equation*}
    \sigma(X,Y)<\sigma(X,C),
  \end{equation*}
  contradicting the maximizing property of $X$ and $Y$.

  By the same token, if $\partial(C)\geq 0$
  we get that 
  \begin{equation*}
    \sigma(X,Y)<\sigma(C, Y).
  \end{equation*}
  Finally, if $\partial(C)<0$ and $\partial(X)=0$, we get
  \begin{equation*}
      \kappa(X,Y)<\kappa(X,C)
  \end{equation*}
  in the same way, which finishes the proof.
\end{proof}

\subsection{Algorithm for non-Eulerian Graphs}

\begin{algorithm}[!ht]
\begin{algorithmic}[1]
  \State{} Choose $X_1,Y_1\subseteq V$ such that
  $\partial(X_1)\geq0$ and $\partial(Y_1)<0$.\label{algline:initial}
  \State{} $i:=1$
  \Loop{}
  \State{}
  $\displaystyle \sigma_i:= \sigma(X_i, Y_i)$\label{algline:sigma_i}
  \State{}
  $\displaystyle \kappa_i:= \kappa(X_i, Y_i)$
  \State{} $C_i \in \argmin\{b(X)+\sigma_i\delta(X)-\kappa_i\partial(X) :
  X\subseteq V\}$\label{algline:xy:compc}
  \If {$b(C_i)+\sigma_i\delta(C_i)-\kappa_i\partial(C_i) \geq 0$}
  \State{} RETURN $\kappa_i, \sigma_i, X_i, Y_i$\label{algline:xy:opt}
  \ElsIf {$\varrho(C_i)+\delta(C_i)=0$}
  \State{} RETURN ``INFEASIBE''\label{algline:xy:infeas}
  \ElsIf  {$\partial(C_i)\geq 0$}
  \State{} $X_{i+1}:=C_i$
  \State{} $Y_{i+1}:=Y_i$
  \Else{} 
  \State{} $X_{i+1}:=X_i$
  \State{} $Y_{i+1}:=C_i$
  \EndIf{} 
  \State{} $i\longleftarrow i+1$
  \EndLoop{}
\end{algorithmic}
\caption{Minimum spread calculation in non-Eulerian graphs}\label{alg:xy}
\end{algorithm}

The proof of Theorem~\ref{thm:eqxy}, immediately provides the
algorithm outlined in Algorithm~\ref{alg:xy} for finding an optimal
pair of sets $X$ and $Y$.

The initial sets $X_1$ and $Y_1$ can be chosen as single node sets
satisfying conditions in line~\ref{algline:initial}.

Computing $C_i$ in line~\ref{algline:xy:compc} involves in the
minimization of the submodular function
$b(X)+\sigma_i\delta(X)-\kappa_i\partial(X)$.  If the algorithm terminates
at line\ \ref{algline:xy:opt}, then the last value of the $\sigma_i$
is optimal and $X_i, Y_i$ are maximizing sets in Theorem
$\ref{thm:eqxy}$.

Note, that $\partial(X_i)\geq 0$ and $\partial(Y_i)<0$ hold in every
iteration, therefore the denominator of $\sigma_i$ and $\kappa_i$ are
not zero.

The algorithm is clearly finite because the value $\sigma_i$ is
monotonically increasing throughout the execution, and either the value
of $\sigma_i$ or $\kappa_i$ strictly increases in each iteration.

In the following, we will show that not only is Algorithm~\ref{alg:xy}
finite, but in fact it runs in strongly polynomial time. Let us
consider two sets $Z_1, Z_2\subseteq V$ \emph{equivalent} if
$\delta(Z_1)=\delta(Z_2)$, $\varrho(Z_1)=\varrho(Z_2)$ and
$b(Z_1)=b(Z_2)$.
\begin{lem}\label{thm:noneqnum}
  The algorithm computes at most ${(m+1)}^2$ non-equivalent sets.
\end{lem}
\begin{proof}
  The lemma is proved by showing that if $C_1$ and $C_2$ are
  non-equivalent sets found by the algorithm and
  $\delta(C_1)=\delta(C_2)$ then $\varrho(C_1)\neq \varrho(C_2)$.

  If $b(C_1)=b(C_2)$, then
  $\varrho(Z_1)\neq\varrho(Z_2)$, because they are
  non-equivalent. Therefore, we may assume that $b(C_1)\neq b(C_2)$.

  First, let us consider the case when
  $\delta(C_1)=\delta(C_2)=0$. Then $\varrho(C_1) \neq 0$ and
  $\varrho(C_2) \neq 0$, otherwise the problem would be infeasible.
  Thus, the expression in line~\ref{algline:xy:compc} can be rewritten
  as
\begin{equation} \label{eq:argmin} \argmin
  \{b(X)+\sigma_i\delta(X)-\kappa_i\partial(X)\}= \argmin
  \bigg\{\varrho(x)\left(\frac{b(X)}{\varrho(X)}-\kappa_i\right)
  \bigg\}.
\end{equation}
By re-indexing we may assume that
$\frac{b(C_1)}{\varrho(C_1)} > \frac{b(C_2)}{\varrho(C_2)}$.
(If these two values are equal then the algorithm will only
find the set with the greater $\varrho$ value.)  Since $C_{1}$
minimizes the right hand side of (\ref{eq:argmin}) for
$\kappa_{1}$, we have $\varrho(C_{i_1}) < \varrho(C_{i_2})$.

Now assume that $\delta(C_1)=\delta(C_2)>0$ and $b(C_1)>b(C_2)$. Then
the expression in line~\ref{algline:xy:compc} can be rewritten as
\begin{equation} \label{eq:argmax} \argmin
  \{b(X)+\sigma_1\delta(X)-\kappa_1\partial(X)\}= \argmax
  \bigg\{
  \delta(x)\left(\frac{\kappa_1\partial(X)-b(X)}{\delta(X)}-\sigma_1\right)
  \bigg\}.
\end{equation}

Since $\delta(C_1)=\delta(C_2)$ and $C_{1}$ maximizes the right hand side of
(\ref{eq:argmax}) for $\kappa_{1}, \sigma_{1}$, then
\begin{equation*}
  \frac{\kappa_1\partial(C_{1})-b(C_{1})}{\delta(C_{1})}
  \geq \frac{\kappa_1\partial(C_{2})-b(C_{2})}{\delta(C_{2})},
\end{equation*}
from which we get that
$\frac{\partial(C_1)}{\delta(C_{1})} >
\frac{\partial(C_2)}{\delta(C_{2})}$. Thus
$\varrho(C_1)\neq \varrho(C_2)$, which finishes the proof.
\end{proof}

\begin{lem}\label{thm:noneqpair}
  Assume that $X_i$ and $X_j$ are equivalent for some $i\neq j$. Then
  $Y_i$ and $Y_j$ are not equivalent.
\end{lem}
\begin{proof}
  If $X_i$ is equivalent to $X_j$ and $Y_i$ is equivalent to $Y_j$,
  then $\sigma_i=\sigma_j$ and $\kappa_i=\kappa_j$ hold.  But --- as
  we have seen in the proof of Theorem~\ref{thm:eqxy} --- the
  $\sigma_i$ is monotonically increasing, and either $\sigma_i$ or
  $\kappa_i$ strictly increases at each iteration.
\end{proof}

Combining Lemma~\ref{thm:noneqnum} and \ref{thm:noneqpair}, we
immediately get the following upper bound on the number of iterations.

\begin{thm}
  Algorithm~\ref{alg:xy} terminates after at most $\mathcal{O}(m^4)$
  iterations. \qedwhite{}
\end{thm}

The running time of each single iteration is dominated by the submodular
function minimization, therefore we get the following.
\begin{thm}
  Algorithm~\ref{alg:xy} computes $\sigma^*$ along with the maximizing
  sets $X$ and $Y$ in $\mathcal{O}(m^4\Upsilon)$ time.
  \qedwhite{}
\end{thm}

\subsection{Improved Algorithm for Non-Eulerian Graphs}

This section presents an $\mathcal{O}(m^2\Upsilon)$ time algorithm for
finding the minimum spread and the corresponding dual pair of sets $X$
and $Y$.

\begin{algorithm}[!ht]
\begin{algorithmic}[1]
  \State{}Choose $X_1,Y_1\subseteq V$ such that
  $\partial(X_1)\geq 0$ and $\partial(Y_1)<0$
  \State{}$i:=1,\quad j:=1$
  \Loop{}
  \State{}$\sigma_i:= \sigma(X_i,Y_j)$
  \State{}$\kappa_i:= \kappa(X_i,Y_j)$
  \State{}$C \in \argmin\{b(X)+\sigma\delta(X)-\kappa\partial(X) :
  X\subseteq V\}$\label{algline:xy2:compc}
  \If{$b(C)+\sigma_i\delta(C)-\kappa_i\partial(C) \geq 0$}
  \State{}RETURN $\kappa_i, \sigma_i, X_i, Y_j$\label{algline:xy2:opt}
  \ElsIf{$\varrho(C)+\delta(C)=0$}
  \State{}RETURN ``INFEASIBLE''\label{algline:xy2:infeas}
  \ElsIf{$\delta(C)=0$}\label{algline:xy2:delta0}
  \State{}$X_1=C$ and $i:=1$ 
  \While{$\sigma(C,Y_j)\leq\sigma(C,Y_{j-1})$}
  \State{}$j := j-1$
  \EndWhile{}
  \Else{}
  \While{$\delta(X_i)=0$ or $\frac{\partial(X_i)}{\delta(X_i)}\leq\frac{\partial(C)}{\delta(C)}$
    or $\kappa(C,X_i)\geq\kappa(C,X_{i-1})$}
  \State{}$i := i-1$
  \EndWhile{}
  \While{$\frac{\partial(Y_J)}{\delta(Y_j)}\geq\frac{\partial(C)}{\delta(C)}$
    or $\kappa(C,Y_j)\leq\kappa(C,Y_{j-1})$}
  \State{}$j := j-1$
  \EndWhile{}
  \If{$\partial(C)\geq 0$}
  \State{}$X_{i+1}:=C$
  \State{}$i := i+1$
  \Else{} 
  \State{}$Y_{j+1}:=C$
  \State{}$j := j+1$
  \EndIf{} 
  \EndIf{} 
  \EndLoop{}
\end{algorithmic}
\caption{Minimum spread calculation}\label{alg:xy2}
\end{algorithm}

The improved procedure is outlined in Algorithm~\ref{alg:xy2}. The
main difference compared to Algorithm~\ref{alg:xy} is that the
algorithm revisits the previously found sets at each iteration,
ensuring that every set can be found at most once.  Therefore, the
number of iterations is equal to the number of non-equivalent sets
found by the algorithm.

For this we maintain sequences of previously calculated sets
$X_1, X_2,\dots,X_i$ and $Y_1,Y_2,\dots,Y_j$ such that
\begin{equation}\label{eq:seqeq}
  \frac{\partial(Y_1)}{\delta(Y_1)}
  < \frac{\partial(Y_2)}{\delta(Y_2)}
  < \dots
  < \frac{\partial(Y_j)}{\delta(Y_j)}
  < \frac{\partial(X_i)}{\delta(X_i)}
  < \dots
  < \frac{\partial(X_2)}{\delta(X_2)} < \frac{\partial(X_1)}{\delta(X_1)}
\end{equation}

At each iteration, the unnecessary sets are deleted. 
Note that --- by Theorem~\ref{thm:submodflow}--- if there exists a set $Z$ with
$\delta(Z)=0$ and $\varrho(Z)\neq 0$ and
$\kappa > \frac{b(Z)}{\varrho(Z)}$, then there exists no feasible submodular
flow with lower bound $\kappa\egy$.
Analogously, if $\delta(Z)>0$ and
$\sigma<\frac{\kappa\partial(Z)-b(Z)}{\delta(Z)}$, then a
$(\kappa\egy, (\kappa+\sigma)\egy)$-bounded submodular flow cannot exist.
Therefore, if the value of a flow is at least $\kappa$ on each edge,
then spread of the flow is at least $\frac{\kappa\partial(Z)-b(Z)}{\delta(Z)}$.

Therefore a set $X_{i}$ is to be
considered unnecessary if for any $\kappa'$ there exists a set
$Z\in \{Y_1,Y_2,\dots ,Y_j, X_{i-1}, \dots , X_1\}$ such that
\begin{equation*}
  \frac{\kappa'\partial(X_{i})-b(X_{i})}{\delta(X_{i})}
  \leq \frac{\kappa'\partial(Z)-b(Z)}{\delta(Z)}.
\end{equation*}
It means that for any $\kappa'$, there exists another found set $Z$
which defines a larger lower bound for the spread than those obtained
by $X_i$. Due to the deleting process at the end of every iteration,
the inequalities in~(\ref{eq:seqeq}) hold true.

It is clear that if
$Z,Z' \in \{Y_1,Y_2,\dots ,Y_j, X_i, X_{i-1}, \dots , X_1\}$ are two
different sets and
$\frac{\partial(Z)}{\delta(Z)} < \frac{\partial(Z')}{\delta(Z')}$ then
the followings holds true.
\begin{align*}
  \frac{\kappa'\partial(Z)-b(Z)}{\delta(Z)}
  < \frac{\kappa'\partial(Z')-b(Z')}{\delta(Z')}
  \quad \text{if } \kappa'<\kappa(Z,Z') \\
  \frac{\kappa'\partial(Z)-b(Z)}{\delta(Z)}
  = \frac{\kappa'\partial(Z')-b(Z')}{\delta(Z')}
  \quad \text{if } \kappa'=\kappa(Z,Z') \\
  \frac{\kappa'\partial(Z)-b(Z)}{\delta(Z)}
  > \frac{\kappa'\partial(Z')-b(Z')}{\delta(Z')}
  \quad \text{if } \kappa'>\kappa(Z,Z') 
\end{align*}
At the beginning of iteration $i$, every set in
$\{Y_1,Y_2,\dots ,Y_j, X_i, X_{i-1}, \dots , X_1\}$ was necessary,
thus
\[\kappa(Y_1, Y_2)< \kappa(Y_2,Y_3)< \cdots < \kappa(Y_j,X_j) < \cdots <
  \kappa(X_2, X_1).\]
Therefore, the following properties hold true in
every iteration.
\begin{enumerate}
\item If $\kappa<\kappa(Y_1,Y_2)$ and
  $Z\in \{Y_2,\dots ,Y_j, X_i, X_{i-1}, \dots , X_1\}$, then
  \begin{equation*}
    \frac{\kappa\partial(Y_1)-b(Y_1)}{\delta(Y_1)}
    > \frac{\kappa\partial(Z)-b(Z)}{\delta(Z)}.
  \end{equation*}
\item If $\kappa(Y_{j'-1},Y_{j'}) < \kappa < \kappa(Y_{j'},Y_{j'+1})$
  and $Z\in \{Y_1,Y_2,\dots ,Y_j, X_i, X_{i-1}, \dots , X_1\}$,
  $Z\neq Y_{j'}$, then
  \begin{equation*}
    \frac{\kappa\partial(Y_{j'})-b(Y_{j'})}{\delta(Y_{j'})}
    > \frac{\kappa\partial(Z)-b(Z)}{\delta(Z)}.
  \end{equation*}
\item If $\kappa(Y_{j-1},Y_{j}) < \kappa < \kappa(Y_{j},X_i)$ and
  $Z\in \{Y_1,Y_2,\dots Y_{j-1}, X_i, X_{i-1}, \dots , X_1\}$, then
  \begin{equation*}
    \frac{\kappa\partial(Y_{j})-b(Y_{j})}{\delta(Y_{j})}
    > \frac{\kappa\partial(Z)-b(Z)}{\delta(Z)}.
  \end{equation*}
\end{enumerate}
Similar properties hold for any $X_{i'}\in \{X_i, X_{i-1}, \dots , X_1\}$.

If $\delta(C)=0$, then $\kappa_i> \frac{b(C)}{\varrho(C)}$. In this
case (line (\ref{algline:xy2:delta0})) every set in
$\{X_i, X_{i-1} \dots X_1\}$ is unnecessary.
At the beginning of this
iteration, $Y_{j'}$ defines the largest lower bound for the spread if
and only if
$\kappa(Y_{j'-1},Y_{j'})\leq \kappa \leq
\kappa(Y_{j'},Y_{j'+1})$. Therefore, $Y_{j'}$ is no longer necessary
if and only if $\frac{b(C)}{\varrho(C)}<
\kappa(Y_{j'-1},Y_{j'})$. This is equivalent to the condition
$\sigma(C,Y_j)\leq\sigma(C,Y_{j-1})$.

If $\delta(C)\neq 0$, we have that
\begin{equation*}
  \frac{\kappa\partial(Y_j)-b(Y_j)}{\delta(Y_j)}
  = \sigma(X_i, Y_j) < \frac{\kappa\partial(C)-b(C)}{\delta(C)}.
\end{equation*}
Therefore, $X_i$ or $Y_j$ are unnecessary whenever
$\frac{\partial(X_i)}{\delta(X_i)}\leq\frac{\partial(C)}{\delta(C)}$
or
$\frac{\partial(Y_j)}{\delta(Y_j)}\geq\frac{\partial(C)}{\delta(C)}$
hold, respectively.  Since $X_i$ was necessary at the end of the
previous iteration, then
\begin{equation*}
  \frac{\kappa\partial(X_i)-b(X_i)}{\delta(X_i)}
  \geq \frac{\kappa\partial(X)-b(X)}{\delta(X)}
  \quad \forall X \in \{ X_{i-1},\dots ,X_1\},
\end{equation*}
holds for any $\kappa(Y_{j},X_i) \leq \kappa \leq \kappa(X_i, X_{i-1})$,
which means that $X_i$ is unnecessary if and only if
\begin{equation} \label{eq:X_ideleting}
  \frac{\kappa\partial(X_i)-b(X_i)}{\delta(X_i)}
  \leq \frac{\kappa\partial(C)-b(C)}{\delta(C)}
  \quad \forall \kappa\in [ \kappa(Y_{j},X_i), \kappa(X_i, X_{i-1})]
\end{equation}
Since
$\frac{\partial(X_i)}{\delta(X_i)}>\frac{\partial(C)}{\delta(C)}$
(otherwise $X_i$ would have already been deleted), it is enough to
check if expression~(\ref{eq:X_ideleting}) holds at the end of the interval,
i.\,e.\ for $\kappa= \kappa(X_i, X_{i-1})$. Thus,
inequality~(\ref{eq:X_ideleting}) reduces to
$\kappa(C,X_i)\geq\kappa(C,X_{i+1})$. The decision of whether or not
$Y_j$ is necessary can be made in the same way. Note that if
$\frac{\partial(X_1)}{\delta(X_1)}>\frac{\partial(C)}{\delta(C)}$,
then $X_1$ will not be deleted. Analogously, if
$\frac{\partial(Y_1)}{\delta(Y_1)}<\frac{\partial(C)}{\delta(C)}$, then
$Y_1$ will not be deleted in the current iteration.

After deleting all the unnecessary sets, the algorithm adds the new
set $C$ to the appropriate subsequence. Note that $C$ is always necessary in the
iteration where it is computed and it is not equivalent to any
$Z\in \{Y_1, Y_2,\dots,Y_j, X_i, X_{i-1}, \dots, X_1\}$. In addition, a deleted (i.\,e.\ unnecessary) set cannot become
necessary again. Therefore, the algorithm will find every set at most
once, and delete it at most once.

To sum up, the algorithm finishes after $\mathcal{O}(m^2)$ iterations,
therefore we get the following theorem.
\begin{thm}
  Algorithm~\ref{alg:xy2} solves the Balanced Submodular Flow Problem
  in $\mathcal{O}(m^2\Upsilon)$ time. 
  \qedwhite{}
\end{thm}

\section{Weighted Balanced Submodular Flows}\label{sec:weighted}

This section introduces the weighted version of the problem and
extends the dual characterization to this case. Then, a strongly
polynomial algorithm is presented to determine the minimum weighted
spread with a positive weight function.

\begin{defn}
  Given a weight function $c: A\nyil \R^+$ on the edges, the
  \emph{weighted spread} of a vector $x\in\R^A$ is the
  value
\[\sigma_c(x):=\max_{a\in A}c(a)x(a) - \min_{a\in A}c(a)x(a)\]
\end{defn}

The \emph{Balanced Submodular Flow Problem} is defined as the following.

\begin{prob}\label{prb:gensubflow}
  Let us given a directed graph $G=(V,A)$, a weight function
  $c: A\nyil \R^+$ and a submodular function $b:\P(V)\nyil\Z$. Find
\begin{equation}
  \sigma^*:=\min\left\{\sigma_c(x):x\in\R^A,
    \varrho_x(X)-\delta_x(X) \leq b(X)\quad\forall X\subseteq V\right\}
\end{equation}
along with a minimizing vector $x$.
\end{prob}

Let us introduce the following notations.
\begin{flalign*}
  \tilde{c} (e) &:= \frac{1}{c(e)}\\
  \varrho_{\tilde{c}}(X)&:=\sum_{e\in \varrho(X)}\tilde{c}(e)\\
  \delta_{\tilde{c}}(X)&:=\sum_{e\in \delta(X)}\tilde{c}(e) \\
  \partial_{\tilde{c}}(X)&:=\varrho_{\tilde{c}}(X)-\delta_{\tilde{c}}(X)
\end{flalign*}  

\begin{defn}
  For an arbitrary real value $\kappa\in \R$, let $s(\kappa)$ denote
  the minimal value $\sigma$ for which there exists a
  $\left(\kappa\tilde{c},(\kappa+\sigma)\tilde{c}\right)$-bounded
  submodular flow $x$.
\end{defn}

Clearly, 
\begin{equation*}
    \sigma^* = \min_{\kappa}s(\kappa)
\end{equation*}

The following theorem determines a dual characterization of
Problem~\ref{prb:gensubflow} and then we will show that the optimal
spread can be computed in strongly polynomial time.
 
\begin{thm}\label{thm:weightedoptimum}
  Assume that there exists a set $X\subseteq V$ such that
  $\partial_{\tilde{c}}(X)\gneqq 0$. Then
  \begin{equation} \label{eq:weighted_opt}
    \sigma^*= \max \Bigg\{
    \frac{-b(X)\partial_{\tilde{c}}(Y)+b(Y)\partial_{\tilde{c}}(X)}
    {\delta_{\tilde{c}}(X)\partial_{\tilde{c}}(Y)-
      \delta_{\tilde{c}}(Y)\partial_{\tilde{c}}(X)}
    :
    \partial_{\tilde{c}}(X)\geq 0, \partial_{\tilde{c}}(Y)<0  \Bigg\},
  \end{equation} 
  or $\sigma^*=0$ if the maximum above is negative.

  If $\partial_{\tilde{c}}(X)=0$ holds for all $X\subseteq V$, then
  $s(\kappa)$ is a constant function. In this case
  \begin{equation}
    \sigma^* = \max_{\partial_{\tilde{c}}(X)=0}\Bigg\{
    \frac{-b(X)}{\delta_{\tilde{c}}(C)} : b(X)<0\Bigg\},
  \end{equation}
  or $\sigma^*=0$ if $b(X)\geq 0$ hold for all $X\subseteq V$.
\end{thm}

The proof of Theorem \ref{thm:weightedoptimum} follows the same pattern as that of
Theorem~\ref{thm:eqxy} and the details are left to the reader.

The following two sections present a strongly polynomial algorithm to
compute the value of $s(\kappa)$ then to find its minimum, therefore
solving the weighted version of the Balanced Submodular Flow
Problem. This latter algorithm also provides a corresponding dual
solution described in Theorem~\ref{thm:weightedoptimum}.

\subsection{Computing the Value of $s(\kappa)$}

In this section, the following theorem is proved.

\begin{thm}\label{thm:s(kappa)computed}
  The value of $s(\kappa)$ can be computed with
  $n^2+\mathcal{O}(m\log^2(m))$ submodular function
  minimizations.
\end{thm}

By Theorem~\ref{thm:submodflow}, we get that
\begin{claim}
  \begin{equation}\label{eq:s(kappa)weighted}
    s(\kappa) = \min \left\{\sigma \geq 0 :
      \min_{X\subseteq V}
      \left(b(S)-\kappa\partial_{\tilde{c}}(X) -
        \sigma\delta_{\tilde{c}}(X)\right)\geq 0\right\}
  \end{equation}
  \qedwhite{}
\end{claim}

Note that $b(S)-\kappa\partial_{\tilde{c}}(X)$ is a submodular function,
therefore Theorem~\ref{thm:s(kappa)computed} is clearly a special case
of the following extension of Theorem~\ref{thm:Goemans-Gupta-Jaillet}
proven below.

\begin{thm}\label{thm:function_value}
  Let $b$ be a submodular function over the set $V$, $|V|=n$ and assume that
  $b(\emptyset)\geq 0$. In addition, let $f:\P(V) \nyil \P(A)$ and
  $a\in \R^A$, such that $b'(S)=a(f(S))=\sum_{e\in f(S)}a(e)$ is
  also submodular and let $|E|=m$. Then the value of
  \begin{equation}\label{eq:delta2*}
    \delta^*:= \min \left\{\delta \geq 0 :
      \min_{S\subseteq V}\left(b(S)+\delta a(f(S))\right)\geq 0\right\}
  \end{equation}
  can be computed by performing $n^2+\mathcal{O}(m\log^2m)$ submodular
  function minimization over the set $V$, i.\,e.~in time
  $\mathcal{O}((n^2+m\log^2m)\Upsilon)$.
\end{thm}

\begin{proof}
Note that $\delta^*$ is the smallest root of the piecewise linear
concave function
\[h(\delta):=\min_{S\subseteq V}\left(b(S)+\delta
    a(f(S))\right).\]
Therefore $\delta^*$ can be computed using the
well-known Discrete Newton Method outlined in
Algorithm~\ref{alg:s_kappa}. In the following, we show that the
number of iterations taken by the algorithm is at most
$n^2+\mathcal{O}(m\log^2m)$.

\begin{algorithm}[!h]
\begin{algorithmic}[1]
  \State{}$\delta_1=0$, $i=0$
  \Loop{}
  \State{}$i= i+1$
  \State{}$h_i = \min\{b(X)+\delta_i a(f(X))\}$
  \State{}$X_i\in \argmin\{b(X)+\delta_i a(f(X))\}$
  \If{$h_i=0$} \Return{}$\delta_i, X_i$\EndIf{}
  \State{}$a_i:=a(f(X_i))$
  \If{$a_i=0$}\label{alg:s_kappa:undefline} \Return{}UNDEFINED, $X_i$\EndIf{}
  \State{}$\delta_{i+1} = \frac{-b(X_{i})}{a_i}$
  \EndLoop{}
\end{algorithmic}
\caption{Discrete Newton Method}\label{alg:s_kappa}
\end{algorithm}

We use the following generic properties of the Discrete Newton Method
(\cite{radzikfrac}).
 
\begin{claim}\label{lem:newtonmon}\,
  \begin{enumerate}
  \item $h_1<h_2<\cdots <h_t=0$
  \item $\delta_1<\delta_2<\dots \delta_t$ 
  \item  $a_1>a_2>\cdots >a_t\geq 0$\qedwhite{}
  \end{enumerate}
\end{claim}

\begin{claim}\label{lem:hi+ai<=1}\,  
  \[\frac{h_{i+1}}{h_i}+\frac{a_{i+1}}{a_i} \leq 1\]
  \qedwhite{}
\end{claim}

Let $I$ denote the set of indices for which
$\frac{a_{i+1}}{a_i}\leq \frac{2}{3}$. Then --- as a direct
consequence of Lemma~\ref{goemans} --- we get that
\[
  |I|=\Big|\left\{ i: \frac{a_{i+1}}{a_i}\leq \frac{2}{3} \right\}
  \Big| = \mathcal{O}(m\log m).
\]

Let $J$ be the set of the remaining indices, i.\,e.
\begin{equation*}
    J=\left\{ i: \frac{a_{i+1}}{s_i} > \frac{2}{3} \right\}.
\end{equation*}
Due to Claim~\ref{lem:hi+ai<=1}, $\frac{h_{i+1}}{h_i}<\frac{1}{3}$
hold for any $i\in J$.

Let us consider the smallest partition
$J=J_1\cup J_2,\cup\cdots\cup J_q$ such that $J_l$ is an interval for
each $l=1,\dots, q$, i.\,e.\ $J_l=[s_l,t_l]$ and $t_l+1<s_{l+1}$. Let
$t'_l:=t_l-\lceil \log(n^2/4)\rceil$ and $J'_l:=[s_l,t'_l]$. (If
$t'_l<s_l$, then let $J'_l:=\emptyset$). Let
$J':=J'_1\cup J'_2,\cup\cdots\cup J'_q$.

\begin{lemma}
  For each $j\in J'$, we have
  \[X_j\notin R(X_k: k\in J',\ j<k).\]
\end{lemma}

\begin{proof}
  On the one hand, let $s_i\leq l < t'_i$ for some $1\leq i \leq
  q$. Then the sequence $X_{l}, X_{l+1}, \dots, X_{t'_i}$ with the
  submodular function $b'(X):=b(X)+\delta_{t_i}a(f(X))$ satisfies the
  conditions of Lemma~\ref{thm:b_decrease_fast}, therefore
  $X_{l}\notin R(X_{l+1}, X_{l+2}, \dots, X_{t'_i})$.
  
  On the other hand, if $i<q$, then let
  $F:=J'_{i+1}\cup\dots\cup J'_q$. Clearly, $b'(X_j)<0$ for each
  $j\in F$. Lemma~\ref{thm:notinringfamily} with $M=0$
  implies\footnote{Note that
    $b'(X_{t_i})=\min\{b'(X): X\subseteq V\}$.} that
  $b'(X)\leq -\frac{n^2}{4}b'(X_{t_i})$ holds for each
  $X\in R(X_k: k\in F)$.  Due to Theorem~\ref{thm: submodfunc},
  $X_{l}\in R(F)$ only if
  $t'_i < l \leq t_i$.  Therefore if $s_i\leq l \leq t'_i$, then
  $X_l\notin R(X_k: k\in J',\quad l<k)$, finishing the proof.
\end{proof}

The sequence $(X_j : j\in J')$ meets the conditions of
Lemma~\ref{thm:chainringfamilies}, thus $|J'|\leq\binom{n+1}{2}+1$.

Finally the total number of iterations taken by
Algorithm~\ref{alg:s_kappa} is
\begin{align*}
  \big|I\big|+\big|J\big|
  &\leq\big|I\big|+\big|J'\big|+2\log(n)|I|\\
  &\leq O(m\log m)+\binom{n+1}{2}+1 + 2\log(n)O(m\log m)\\
  &= O(n^2+m\log^2m),
\end{align*}
finishing the proof of Theorem~\ref{thm:s(kappa)computed}.
\end{proof}

Note that when Algorithm~\ref{alg:s_kappa} is used to compute the value of
$s(\kappa_0)$ by finding the smallest root of the function
$h(\sigma):=\min_{X\subseteq V}\left(b(S)-\kappa\partial_{\tilde{c}}(X)
  -\sigma\delta_{\tilde{c}}(X)\right)$, and the algorithm terminates at
line~\ref{alg:s_kappa:undefline}, then the provided set $X\subseteq V$
will have the properties $\delta_{\tilde{c}}(X)=0$ and
$\kappa_0> \frac{b(X)}{\varrho_{\tilde{c}}(X)}$. Therefore, no
submodular flow with lower bound $\kappa_0\tilde{c}$ may exist.

With this information in hand, one can compute the largest value
$\kappa_{\max}$ for which $s(\kappa_{\max})$ is defined (i.\,e.\ there
exists a submodular flow with lower bound $\kappa_{\max}\tilde{c}$) as
follows.

By Theorem~\ref{thm:submodflow},
\begin{equation*}
  \kappa_{\max}=
  \max\bigg\{\kappa : \min\{b(X)-\kappa\varrho_{\tilde{c}}(X):
  \forall X\subseteq V,\quad\delta_{\tilde{c}}(X)=0\} \geq 0 \bigg\}.
\end{equation*}

Therefore, this value can again be computed by applying the Discrete
Newton Method to find the largest root of the concave function
\[h(\kappa):=\min\{b(X)-\kappa\varrho_{\tilde{c}}(X): \forall
  X\subseteq V,\quad\delta_{\tilde{c}}(X)=0\}\]
with initial value
$\kappa_0$.

By the same token as in the proof of
Theorem~\ref{thm:s(kappa)computed}, the value of $\kappa_{\max}$ can be
found in $n^2+\mathcal{O}(m\log^2(m))$ iterations.

\subsection{Minimizing $s(\kappa)$}

\begin{claim}\label{lem:sconvex}
The function of $s(\kappa)$ is a convex function.
\end{claim}
\begin{proof}
  For any $\kappa_1,\kappa_2\in\R$, let $x_1$ and $x_2$ be submodular
  flows such that
  $\kappa_i\egy\leq x_i \leq (\kappa_i+s(\kappa_i))\egy$, and let
  $0\leq\lambda\leq 1$. Then $x':=\lambda x_1 + (1-\lambda)x_2$ is
  also a submodular flow and
  $[\lambda\kappa_1 + (1-\lambda)\kappa_2]\egy\leq x'\leq
  [\lambda(\kappa_1+s(\kappa_1)) +
  (1-\lambda)(\kappa_2+s(\kappa_2))]\egy$, therefore
  \begin{equation}
    s(\lambda\kappa_1 + (1-\lambda)\kappa_2)
    \leq \lambda s(\kappa_1) + (1-\lambda) s(\kappa_2)
    \label{eq:sdef}.
  \end{equation}
\end{proof}

Since $s(\kappa)$ is a convex function, the Handler--Zang method ---
detailed in Appendix~\ref{subsec:handler-zang} --- is applicable to
find the minimum of it. It will be shown that it actually provides a
strongly polynomial algorithm. For this an efficient way to find an
appropriate initial interval is given, then a strongly polynomial bound
on the number of iterations is provided.

\subsubsection{Finding the Initial Interval}
In order to find the initial interval, let us start with an arbitrary value
$\kappa_0\leq\kappa_{\max}$ and iterate the usual Newton steps on the
function $s(\kappa)$ until either
\begin{enumerate}
\item a root of $s(\kappa)$ is found. Then the optimal solution of
  $0$ to the original problem is found.
\item a value with a subgradient $0$ is found. In this case the
  minimum of $s(\kappa)$ has also been reached.
\item a value with a subgradient of the opposite sign is found. Then
  $[\kappa_{i-1},\kappa_i]$ (or $[\kappa_i,\kappa_{i-1}]$ if
  $\kappa_i < \kappa_{i-1}$) is an appropriate initial interval.
\item $\kappa_i>\kappa_{\max}$. Then $[\kappa_{i-1}, \kappa_{\max}]$ is
  an appropriate initial interval.
\end{enumerate}

\subsubsection{Upper Bound on the Number of Iterations}

\begin{thm}
  The minimum of $s(\kappa)$ can be computed using the Handler--Zang
  method with $\mathcal{O}(n^4m^6\log^6(m))$ iterations.
\end{thm}

Note that during the execution of the algorithm, a set $X$ with
$\delta_{\tilde{c}}(X)=0$ can never be relevant when computing
(\ref{eq:s(kappa)weighted}).  Therefore it can be rewritten as
follows.
\begin{equation}\label{eqn:eskappa-frac}
  s(\kappa)= \max\bigg\{  \kappa \frac{\partial_{\tilde{c}}(X)}
  {\delta_{\tilde{c}}(X)} -
  \frac{b(X)}{\delta_{\tilde{c}}(X)}
  :
  \delta_{\tilde{c}}(X)>0 \bigg\}.
\end{equation}

Thus,
for any given value of $\kappa$ and set $X$ maximizing
Equation~(\ref{eqn:eskappa-frac}), the value
$\frac{ \partial_{\tilde{c}}(X)}{\delta_{\tilde{c}}(X)}$ is a
subgradient of $s(\kappa)$.

Now, let us assume that $\sigma^*$ is already known, and we are only
looking for the value of $\kappa^*$ for which $s(\kappa^*)=\sigma^*$,
i.\,e.\ we are looking for a root of the function
$\overline{s}(\kappa):=s(\kappa)-\sigma^*$. This could also be
computed by the Discrete Newton Method outlined in
Algorithm~\ref{alg:s_kappavonas}.

An important observation is the following. Let
$a'_1< a'_2< \cdots < a'_k=\kappa^*$ be all the \emph{distinct} values
computed by the Handler--Zang method during the minimization of
$s(\kappa)$ and let
$a'_1=\kappa_0<\kappa_1< \cdots < \kappa_l=\kappa^*$ be the values
computed by the Discrete Newton Method for finding the root of
$\overline{s}(\kappa)$ with the initial value $\kappa_0=a_1$.  Then
$\kappa_i\leq a'_i$ holds for all $i=1,\dots,k$, consequently $k<l$.The same holds for the sequence of
$b'_1 > b'_2 > \cdots > b'_r=\kappa^*$.

The Handler-Zang method
updates either $a_i$ or $b_i$ in every iteration, therefore we get
that the number of its iteration is at most twice the number of
iterations taken by the Discrete Newton Method to find the root of
$\overline{s}(\kappa)$. An upper bound on the latter quantity is given
below.

\begin{thm}
  The root of $\overline{s}(\kappa)$ can be computed using the
  Discrete Newton Method with $\mathcal{O}(n^4m^6\log^6(m))$
  iterations.
\end{thm}

\begin{algorithm}[!ht]
\begin{algorithmic}[1]
  \State{}$\kappa_1=a_1$, $i=1$
  \State{} $\sigma_i = s(\kappa_1)-\sigma^*$
  \While{$\sigma_i>0$}
  \State{} $i = i+1$
  \State{} $\kappa_{i} = \frac{b_{i-1}}{\partial_{i-1}}$
  \State{} $\sigma_{i} = s(\kappa_{i-1})-\sigma^*$
  \State{} $A_{i}= \argmax_{A\subseteq V}
  \frac{\kappa_i \partial_{\tilde{c}}(A)-b(A)}  {\delta_{\tilde{c}}(A)}$
  \State{} $b_i=b(A_i),
  \quad \partial_i=\partial_{\tilde{c}}(A_i),
  \quad \delta_i=\delta_{\tilde{c}}(A_i),
  \quad s_i=\frac{\partial_i}{\delta_i}$
  \EndWhile{}
\end{algorithmic}
\caption{Discrete Newton Method for
  $\overline{s}(\kappa)$}\label{alg:s_kappavonas}
\end{algorithm}

\begin{claim}\label{thm:lessthan1}
  \begin{equation}
    \frac{\sigma_{i+1}}{\sigma_i}+\frac{s_{i+1}}{s_i}\leq 1
    \label{eq:lessthan1}
  \end{equation}
  \qedwhite{}
\end{claim}

An upper bound on the number of iterations is given below. For this
purpose, the indices are divided into parts, depending on the value
of $\frac{s_{i+1}}{s_i}$. Note that it is between $0$ and $1$, since
$|s_{i}|$ is decreasing.

\begin{lem}\label{thm:I1_size}
There are $\mathcal{O}(m^4\log^4(m))$ iterations such that
$\frac{s_{i+1}}{s_i}\leq \frac{2}{3}$.
\end{lem}
\begin{proof}
  Let $i_1, i_2, \dots, i_k$ be the indices for which
  $\frac{s_{i_j+1}}{s_{i_j}}\leq \frac{2}{3}$ holds and consider the sequence
  $\delta_{i_1},\delta_{i_2},\dots ,\delta_{i_k}$.  By
  Theorem~\ref{thm:erdos-szekeres}, there exists a monotonic
  subsequence $\delta_{i'_1}, \delta_{i'_2},\dots , \delta_{i'_{k'}}$
  of length $k'=\left\lfloor\sqrt{k}\right\rfloor$.

  Let us assume that all the subgradients computed by the Newton
  method are positive i.\,e. $\partial_i> 0$. (The case of
  negative subgradients is proved the same way.)

  \paragraph{Case 1. $\delta_{i'_1}\geq \delta_{i'_2} \geq\cdots \geq \delta_{i'_{s'}}$.}
  Then
  \begin{equation*}
    \partial_{i'_{j+1}}
    = s_{i'_{j+1}}\delta_{i'_{j+1}}
    \leq s_{i'_{j+1}}\delta_{i'_{j}}
    \leq \frac{2}{3}s_{i'_{j}}\delta_{i'_{j}}
    = \frac{2}{3}\partial_{i'_j}
  \end{equation*}
  therefore
  \begin{equation*} \partial_{i'_{j+2}}\leq
    \frac{2}{3}\partial_{i'_{j+1}} \leq
    \frac{4}{9}\partial_{i'_{j}} \leq
    \frac{1}{2}\partial_{i'_j}
  \end{equation*}
  holds for all $j=1,\dots,k'-2$. Observe, that with
  \begin{align*}
    b_i&=\tilde{c}(e_i) \quad 1\leq i \leq |E| \\ 
    x_i&= \begin{cases}
      1 &e_i\in \varrho(A_i)\\
      -1 &e_i\in \delta(A_i)\\
      0 &\text{otherwise}
    \end{cases}
  \end{align*}
  the values of $\partial_i$ are of the form required by Lemma~\ref{goemans}.
  Therefore we get $k'=\mathcal{O}(m\log(m))$.
  
  \paragraph{Case 2. $\delta_{i'_1}\leq \delta_{i'_2} \leq\cdots \leq \delta_{i'_{k'}}$.}
  Similarly to Case 1, one can prove that there are
  $\mathcal{O}(m\log(m))$ indices such that
  $\frac{4}{3}\delta_{i'_{j'}} \leq\delta_{i'_{j'+1}}$.

  Now let us consider a subsequence of consecutive indices
  $j_1, j_2,\dots, j_h$ such that
  $\frac{4}{3}\delta_{j_{t}} > \delta_{j_{t+1}}$. Now we get that

  \begin{equation*}
    \frac{8}{9}\partial_{j_t}\geq  \partial_{j_{t+1}}
  \end{equation*}
  holds for all $1\leq t \leq h-1$. By Lemma~\ref{goemans}, we get that
  the length of such a subsequence is
  $\mathcal{O}(m\log(m))$.  Therefore the total length $k'$ of the
  whole sequence is
  $\mathcal{O}(m\log(m)\cdot~ m\log(m))=\mathcal{O}(m^2\log(m^2))$.
\end{proof}

Now let us consider a sequence of consecutive iterations
$i, \dots, i+l$ such that $\frac{s_{j+1}}{s_j}\geq \frac{2}{3}$
for $j=i, \dots, i+l$. Note that by Claim~\ref{thm:lessthan1}
$\frac{\sigma_{j+1}}{\sigma_j}\leq \frac{1}{3}$ also holds for these
indices.

\begin{claim}\label{thm:kappadistance}
  For any $j=i,\dots,i+l-1$, we have
  \begin{equation}
    \kappa_{j+1}-\kappa_{i+l}\leq \frac{1}{2}(\kappa_{j}-\kappa_{i+l})
    \label{eq:kappadistance}.
  \end{equation}
\end{claim}
\begin{proof}
  The claim can easily be proved by the fact that
  \begin{align*}
    \kappa_{j'}-\kappa_{j'+1}
    = \frac{\sigma_{j'}}{s_{j'}}
    =\frac{\sigma_{j'-1}}{s_{j'-1}}
    \cdot\frac{\sigma_{j'}}{\sigma_{j'-1}}\cdot\frac{s_{j'-1}}{s_{j'}}
    \leq \frac{\sigma_{j'-1}}{s_{j'-1}}\cdot \frac{1}{3}\cdot\frac{3}{2}
    = \frac{1}{2} (\kappa_{j'-1}-\kappa_{j'})
  \end{align*}
    holds for each $j'=i,\dots,i+l-1$.
\end{proof}

\begin{claim}\label{fast_decrease}
\begin{equation*} We have
  F_j\geq 2 F_{j+2} \quad \forall i\leq j \leq i+l-2,
\end{equation*}
where
\begin{equation*}
  F_j:= \sigma_{i+l}-(\sigma_j-(\kappa_{j}-\kappa_{i+l})s_{j}).
\end{equation*}
\end{claim}
\begin{proof}
  Note that
  \begin{align*}
    &\sigma_j-(\kappa_{j}-\kappa_{j+1})s_{j}
      = \sigma_j-\frac{\sigma_j}{s_j}s_{j}=0, \\
    &\sigma_{j}\geq \sigma_{i+l}\geq 0,
  \end{align*}
  and
  \begin{align*}
    &s_{j}\geq s_{j+1}
  \end{align*}
  hold for all $j=i,\dots, i+l$. Then the claim is proved as follows:
  \begin{align*}
    F_j= & \sigma_{i+l}-(\sigma_j-(\kappa_{j}-\kappa_{i+l})s_{j})
           = \sigma_{i+l}-\sigma_j+(\kappa_{j}-\kappa_{i+l})s_{j} \\
    =&\sigma_{i+l}-(\sigma_j-(\kappa_{j}-\kappa_{j+1})s_{j})
       +(\kappa_{j+1}-\kappa_{i+l})s_{j}
       \geq (\kappa_{j+1}-\kappa_{i+l})s_{j} \\
    \geq& 2(\kappa_{j+2}-\kappa_{i+l})s_{j}
          \geq  2(\kappa_{j+2}-\kappa_{i+l})s_{j+2} \\
    \geq& 2( \sigma_{i+l}-\sigma_{j+2})
          + 2(\kappa_{j+2}-\kappa_{i+l})s_{j+2}\\
    =&2(\sigma_{i+l}-\sigma_{j+2}+ (\kappa_{j+2}-\kappa_{i+l})s_{j+2})
       = 2F_{j+2}.
\end{align*}
\end{proof}

\begin{claim} We have 
  \begin{equation}\label{eq:F_definition}
      \displaystyle F_j =
  \frac{\sigma_{i+l}\delta_j+b_j-\kappa_{i+l}\partial_j}{\delta_j}.
  \end{equation}
  Furthermore, the expression
  \[
    f(X):=\sigma_{i+l}\delta_{\tilde{c}}(X)+b(X)
    -\kappa_{i+l}\partial_{\tilde{c}}(X)
  \]
  appearing in the numerator in \eqref{eq:F_definition} is a submodular function.
\end{claim}

\begin{proof}
  \begin{align*}
    F_j=& \sigma_{i+l}-(\sigma_j-(\kappa_{j}-\kappa_{i+l})s_{j})
          = \sigma_{i+l}-\frac{\partial_j}{\delta_j}\kappa_j
          +\frac{b_j}{\delta_j}+\kappa_{j}s_j-\kappa_{i+l}s_j \\
    =&\sigma_{i+l}+\frac{b_j}{\delta_j}
       -\kappa_{i+l}(\frac{\partial_j}{\delta_j})
       = \frac{\sigma_{i+l}\delta_j+b_j-\kappa_{i+l}\partial_j}{\delta_j}
  \end{align*}
  and
  \begin{align*}
    f(X)=&\sigma_{i+l}\delta_{\tilde{c}}(X)+b(X)
           -\kappa_{i+l}\partial_{\tilde{c}}(X)
           + \sigma_{i+l}\delta_{\tilde{c}}(Y)+b(Y)
           -\kappa_{i+l}\partial_{\tilde{c}}(Y)  \\
    =&\sigma_{i+l}(\delta_{\tilde{c}}(X)
       +\delta_{\tilde{c}}(Y))+(b(X)+b(Y))
       -\kappa_{i+l}(\partial_{\tilde{c}}(X)+\partial_{\tilde{c}}(Y))\\
    \geq & \sigma_{i+l}(\delta_{\tilde{c}}(X\cup Y)
           +\delta_{\tilde{c}}(X\cap Y))+(b(X\cup Y)+b(X\cap Y))\\
         &-\kappa_{i+l}(\partial_{\tilde{c}}(X\cup Y)
           +\partial_{\tilde{c}}(X\cap Y)) \\
    =&\sigma_{i+l}\delta_{\tilde{c}}(X\cup Y)+b(X \cup Y)
       -\kappa_{i+l}\partial_{\tilde{c}}(X\cup Y) \\
         &+\sigma_{i+l}\delta_{\tilde{c}}(X\cap Y)
           +b(X\cap Y)-\kappa_{i+l}\partial_{\tilde{c}}(X\cap Y) 
  \end{align*}
\end{proof}

\begin{thm}\label{thm:consecutive}
  Let $i, i+1, \dots , i+l$ be consecutive indices such that
  $\frac{\sigma_{j+1}}{\sigma_j}\leq \frac{1}{3}$. Then $l$ is at most
  $\mathcal{O}(n^4m^2\log^2(m))$.
\end{thm}
\begin{proof}
  The same technique is applicable here as in
  Lemma~\ref{thm:I1_size}. Considering the sequence of
  $\delta_i, \delta_{i+1}, \dots , \delta_{i+l}$ and using
  Theorem~\ref{thm:erdos-szekeres}, a monotonic subsequence
  \[
    \delta_{j_1}, \delta_{j_2},\dots, \delta_{j_{l'}}
  \]
  of length
  $\lfloor\sqrt{l}\rfloor$ must exist.

\paragraph{Case 1. $\delta_{j_1}\geq \delta_{j_2}\geq\cdots \geq \delta_{j_{l'}}$.}

  Let us recall that
  $\delta_{j_{k}}\geq \frac{1}{\sqrt{2}}\delta_{j_{k+2}}$ and
  \[\frac{f(A_{j_{k}})}{\delta_{j_{k}}}=F{j_{k}}\geq 2 F_{j_{k+2}}=2
    \frac{f(A_{j_{k+2}})}{\delta_{j_{k+2}}},\]
  therefore
  \[f(A_{j_{k}})\geq f(A_{j_{k+2}})\]
  holds for all $t\leq k \leq t+r-2$.

  Then we can use Theorem~\ref{thm: submodfunc} with
  $T_0=A_{j_{t+r}}, T_1= A_{j_{t+r-4}},\dots, T_k=A_{j_{t+r-4k}}$.
  The conditions of theorem~\ref{thm: submodfunc} hold true, firstly because 
  \begin{align*}
    f(A_{j_{k}})=
    &\sigma_{j_{t+r}}\delta_{j_k}+b_{j_k}-\kappa_{j_{t+r}}\partial_{j_k}
      = \delta_{j_k}\big(\sigma_{j_{t+r}}-
      (\sigma_{j_k}-(\kappa_{j_{k}}-\kappa_{j_{t+r}})s_{j_k})\big) \\
    >& \delta_{j_k}\big(\sigma_{j_{t+r}}
       -(\sigma_{j_k}+(\kappa_{j_k}-\kappa_{j_{k+1}})s_{j_k})\big)
       \geq \delta_{j_k}\big(\sigma_{j_{t+r}}-0\big)>0
  \end{align*}
  hold for all $t\leq k \leq t+r-2$, and
  \begin{align*}
    f(T_0)=
    &f(A_{j_{t+r}})=\sigma_{j_{t+r}}\delta_{j_{t+r}}
      +b_{j_{t+r}}-\kappa_{j_{t+r}}\partial_{j_{t+r}} \\ 
    =&-b_{j_{t+r}}+\kappa_{j_{t+r}}\partial_{j_{t+r}}
       +b_{j_{t+r}}-\kappa_{j_{t+r}}\partial_{j_{t+r}}=0
  \end{align*}
  holds for $k=t+r$, therefore $f(T_0)=f_{\min}$ and
  $f(T_1)=f(A_{j_{t+r-4}})>0=-2f_{\min}$.
  
  Secondly, because
  \begin{align*}
    f(T_k)=f(A_{j_{t+r-4k}})\geq 4 f(A_{j_{t+r-4k+4}})
    = 4 f(A_{j_{t+r-4(k-1)}})= 4 f(T_{k-1} )
  \end{align*}
  Due to Theorem~\ref{thm: submodfunc}, we have $r$ is at most
  $4\binom{n+1}{2}+1=\mathcal{O}(n^2)$.
  
  In conclusion, the length of increasing subsequence $\sqrt{l}=\mathcal{O}(n^2)$, which means that the statement of the theorem holds true in
  this case.

  \paragraph{Case 2.  $\delta_{j_1}\leq \delta_{j_2}\leq\cdots \leq \delta_{j_{l'}}$.}
  First, let $j_{t_1}, j_{t_2}, \dots, j_{t_q}$ denote the indices
  such that $\sqrt[4]{2} \delta_{j_t}\leq \delta_{j_{t+1}}$. Then
  $\delta_{j_{t_{k-4}}}\leq \frac{1}{2} \delta_{j_{t_k}}$ holds for
  all $5\leq k \leq q $. Since $\delta_j$ is linear, using
  Lemma~\ref{goemans}, it follows that $q$ is at most
  $\mathcal{O}(m\log(m))$.

  On the other hand, let $j_t, j_{t+1}, \dots , j_{t+r}$ be a
 subsequence of \emph{consecutive} indices such that
  $\sqrt[4]{2} \delta_{j_t}> \delta_{j_{t+1}}$.  Recall that
  $\delta_{j_{k}}\geq \frac{1}{\sqrt{2}}\delta_{j_{k+2}}$ and
  \[
    \frac{f(A_{j_{k}})}{\delta_{j_{k}}}=F{j_{k}}\geq 2 F_{j_{k+2}}=2
    \frac{f(A_{j_{k+2}})}{\delta_{j_{k+2}}},
  \]
  therefore
  \[
    f(A_{j_{k}})\geq \sqrt{2} f(A_{j_{k+2}})
  \]
  hold for all $t\leq k \leq t+r-2$.
  
  Then we can use Theorem~\ref{thm: submodfunc} with
  $T_0=A_{j_{t+r}}, T_1= A_{j_{t+r-8}},\dots, T_k=A_{j_{t+r-8k}}$
  similarly to the first case and we get that $r$ is at most
  $8\binom{n+1}{2}+1=\mathcal{O}(n^2)$.
  
  In conclusion, the length of increasing subsequence $\sqrt{l}$ is at
  most $\mathcal{O}(n^2m\log(m))$, which proves the theorem.
\end{proof}

Finally, it follows from Lemma~\ref{thm:I1_size} and
Theorem~\ref{thm:consecutive} that the number of iterations the usual Discrete Newton Method is
strongly polynomial. Thus, the minimum of $s(\kappa)$ can be
computed with $\mathcal{O}(n^4m^8\log^6(m))$ submodular
function minimization.  In summary, the following theorem is proved
and an algorithm for finding $\sigma^*$ is obtained.

\begin{thm}
  The value of $\sigma^*$ along with maximizing sets
  $X,Y\subseteq V$ can be computed in
  $\mathcal{O}(n^4m^6\log^6(m)\cdot(n^2+m\log^2(m)))$ submodular function
  minimization problems.
\end{thm}

\section{Balanced Integral Submodular Flows}\label{sec:integral}

A simple example of a graph having two nodes and two parallel edges
shows that the minimum spread solution is not always possible to be
chosen to be integral, even in case of simple network flows with
integer supply vector. This section shows how an integral flow of
minimum (unweighted) spread can be found.

From now on, let us assume that the submodular function $b$ is integral.

Just as in the weighted case, let $s(\kappa)$ be defined as the
minimum value $\sigma$ for which there exists a
$\left(\kappa\egy,(\kappa+\sigma)\egy\right)$-bounded submodular flow
$x$. Note that $s(\kappa)$ can be computed by solving $\mathcal{O}(|A|)$
supermodular function minimization problems for any given $\kappa$.
Clearly,
\begin{equation*}
    \sigma^*=\min_{\kappa} s(\kappa)
\end{equation*}
Note, that the function $s(\kappa)$ is a convex and piecewise linear.

\begin{defn}
  For an arbitrary integer $\kappa\in\Z$, let $s_I(\kappa)$
  denote the minimum value $\sigma$ for which there exists an
  $\left(\kappa\egy,(\kappa+\sigma)\egy\right)$-bounded
  \emph{integral} submodular flow $x\in\Z^A$.
\end{defn}

\begin{claim}
For any $\kappa\in\Z$, we have $s_I(\kappa)=\lceil s(\kappa) \rceil$.
\end{claim}
\begin{proof}
  Clearly, $s_I(\kappa)\geq s(\kappa)$.  On the other hand, the
  definition of $s(\kappa)$ implies the existence of a
  $\left(\kappa\egy,\left(\kappa+s(\kappa)\right)\egy\right)$-bounded
  submodular flow $x$. This flow is also bounded by the integer
  vectors $\kappa\egy$ and
  $\left(\kappa+\lceil s(\kappa)\rceil\right)\egy$, therefore ---
  because of the integrality of the submodular flow
  polyhedron~\cite{Edmonds-Giles} --- an integer submodular flow must
  also exist between these bounds.
\end{proof}

The claim above and the convexity of $s(\kappa)$ immediately give the
following.
\begin{claim} The spread $\sigma_I^*$ of the balanced integral
  submodular flow can be computed as
  \[\sigma_I^*=\min\left\{s_I(\lfloor \sigma^*\rfloor),
      s_I(\lceil \sigma^*\rceil)\right\}
    =\min\left\{\lceil s(\lfloor \sigma^*\rfloor)\rceil,\lceil s(\lceil
      \sigma^*\rceil)\rceil\right\}\].
  \qedwhite{}
\end{claim}

\section{Conclusions}
In this paper, we extended the Balanced Network Flow Problem examined by
Scutell\`a and Klinz~\cite{Scutella, bettina99:_balan} to submodular
flows. A min-max formula and strongly polynomial algorithm were
presented to the problem as well as to its weighted version. Finally,
a strongly polynomial algorithm for Balanced Integral
Submodular Flow Problem was given.

\subsection*{Acknowledgement}
The work was supported by the Lendület Programme of the Hungarian
Academy of Sciences – grant number LP2021--1/2021 and by the Hungarian
National Research, Development and Innovation Office – NKFIH, grant
number FK128673.

The authors would like to acknowledge the valuable
suggestions of András Frank and Kristóf Bérczi.

\printbibliography[heading=bibintoc]

\section*{Appendix}

\appendix

\section{Handler--Zang Method for Convex Function\\
  Minimization}\label{subsec:handler-zang}

This section describes a natural approach to minimize a single
variable convex function. In the case of piecewise linear functions, it is
able to find the exact optimum in a finite number of iterations, which
makes it a useful tool for solving certain parametric combinatorial
optimization problems. Its first use in this scenario is probably due
to Handler and Zang~\cite{handler-zang}.

Let $f:\R\nyil\R$ be a convex function, let $\partial f(x)$ denote the set of
its subgradients at $x$, and let us consider Algorithm~\ref{alg:handlerzang}.
An iteration of the algorithm is shown in Figure~\ref{fig:hziter}.

\begin{algorithm}
  \begin{algorithmic}[1]
    \State{}Let $a_1, b_1\in \R$ such that $a_1\leq \arg\min f \leq b_1$
    \State{}Let $\alpha_1\in \partial f(a_1)$
    and $\beta_1\in \partial f(b_1)$
    \State{}$i:=1$
    \Loop{}
    \State{}$c_i:=\frac{f(b_i)-f(a_i)+\alpha_i a_i
      - \beta_i b_i}{\alpha_i-\beta_i}$
    \State{}$\sigma_i:=\frac{\alpha_i f(b_i)-\beta_i f(a_i)
      +\alpha_i\beta_i(a_i-b_i)}{\alpha_i-\beta_i}$
    \If{$f(c_i)=\sigma_i$}  RETURN $c_i$
    \EndIf{}
    \State{}Let $\gamma_i\in\partial f(c_i)$
    \If{$\gamma_i<0$}
    \State{}$a_{i+1}:=c_i$, $\alpha_{i+1}:=\gamma_i$,
    \State{}$b_{i+1}:=b_i$, $\beta_{i+1}:=\beta_i$
    \Else{} 
    \State{}$a_{i+1}:=a_i$, $\alpha_{i+1}:=\alpha_i$,
    \State{}$b_{i+1}:=c_i$, $\beta_{i+1}:=\gamma_i$
    \EndIf{} 
    \State{}$i\longleftarrow i+1$
    \EndLoop{}
  \end{algorithmic}
  \caption{Handler-Zang method}\label{alg:handlerzang}
\end{algorithm}

\begin{claim} The following statements are true
\begin{itemize}
    \item $a_1\leq a_2\leq a_3\leq\cdots$
    \item $b_1\geq b_2\geq b_3\geq\cdots$
    \item $\alpha_1\leq \alpha_2\leq \alpha_3\leq\cdots$
    \item $\beta_1\leq \beta_2\leq \beta_3\leq\cdots$
    \item For each iteration $i$ either of the following holds
    \begin{enumerate}
        \item $a_i < a_{i+1}$ and $\alpha_i<\alpha_{i+1}$
        \item $b_i > b_{i+1}$ and $\beta_i>\beta_{i+1}$
    \end{enumerate}
  \item The subgradients $\gamma_i\quad (i=1,2,\dots)$ computed during
    the execution are all different.
\end{itemize}
\end{claim}

\begin{figure}[htb!]
\begin{center}
\includegraphics[scale=0.5]{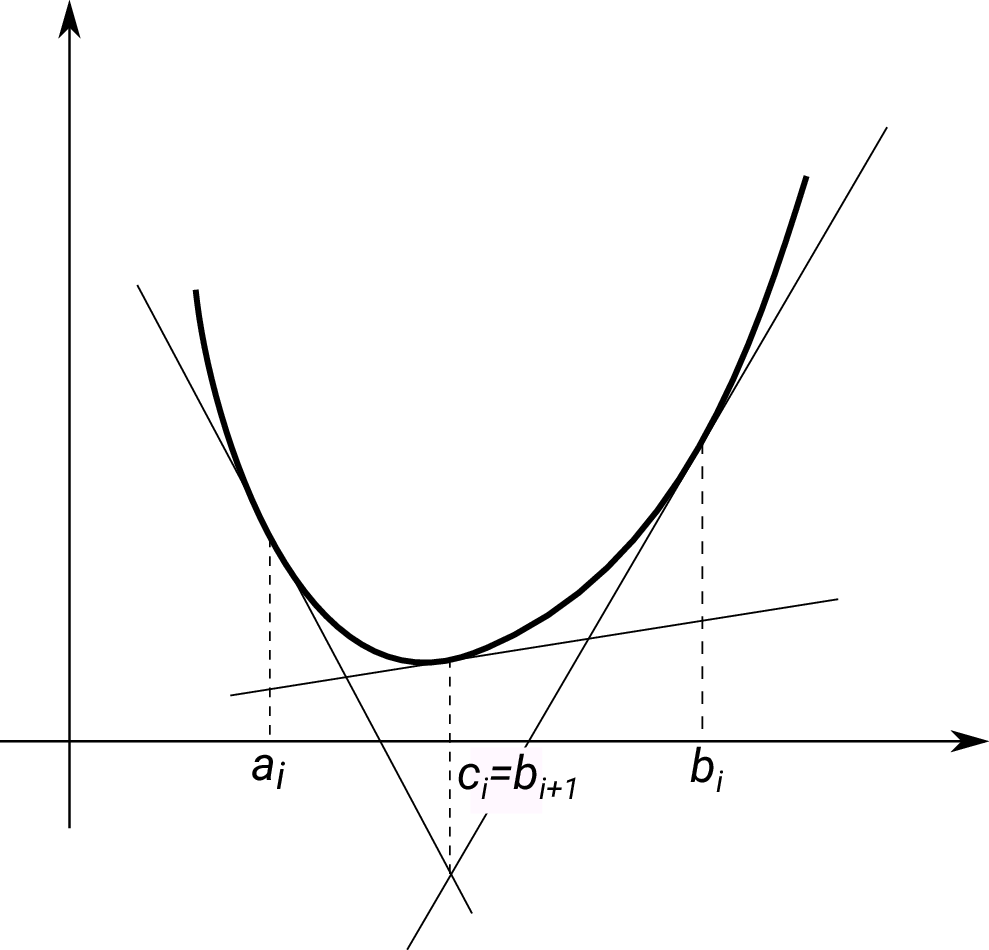}
\end{center}
\caption{An iteration of the Handler-Zang method}\label{fig:hziter}
\end{figure}

\begin{claim}
  If $f$ is a piecewise linear convex function, then
  Algorithm~\ref{alg:handlerzang} finds the minimum of $f$ in finite
  number of steps, and the number of iterations is at most the number
  of linear segments of $f$.
  \qedwhite{}
\end{claim}

In applications in combinatorial optimization, the function $f$ is
typically given implicitly in the form of
\[
  f(x):=\max\{a_i+b{_i}x : i\in{\cal I}\}
\]
by a subroutine that is able to find this maximum for any given value
of $x$. (The index set $\cal I$ may be exponentially large). Note that
for the maximizing $i\in{\cal I}$, the value $b_i$ is a subgradient of
$f(x)$.

Even when the function has exponentially many linear segments, a strongly
polynomial bound on the number of iterations can be given for certain
classes convex functions, see e.\,g.~\cite{onresource-japan}.

\end{document}